\documentclass[12pt,reqno]{amsart}
\usepackage[nospace,noadjust]{cite}

\usepackage{verbatim,amscd,amssymb}
\usepackage{graphicx}
\usepackage{amsmath}
\usepackage{amsmath}
\usepackage{amsfonts}
\usepackage{amssymb}
\usepackage{inputenc}

\usepackage{enumitem}
\usepackage{float}
\usepackage[active]{srcltx}

\graphicspath{ {./images/} }

\usepackage{fontenc}

\usepackage{float}

\newtheorem{theorem}{Theorem}[section]

\theoremstyle{definition}
\newtheorem{definition}[theorem]{Definition}
\newtheorem{example}[theorem]{Example}

\theoremstyle{remark}
\newtheorem{remark}[theorem]{Remark}

\theoremstyle{assumption}

\numberwithin{equation}{section}

\newcommand{\uc}{\mathbb{S}}

\newcommand{\sha}{\succ\mkern-14mu_s\;}

\begin{document}
	
	\date{\today}
	
	\title[Twist like behavior in non-twist patterns of triods]{Twist like behavior in non-twist patterns of triods}

	\author{Sourav Bhattacharya and Ashish Yadav}

	\address[Dr. Sourav Bhattacharya and Ashish Yadav]
{Department of Mathematics, Visvesvaraya National Institute Of Technology Nagpur, 
	Nagpur, Maharashtra 440010, India}
\email{souravbhattacharya@mth.vnit.ac.in}


	
	\begin{abstract}
		
	We prove a sufficient condition for a \emph{pattern} $\pi$ on a \emph{triod} $T$ to have \emph{rotation number} $\rho_{\pi}$ coincide with an end-point of its \emph{forced rotation interval} $I_{\pi}$. Then, we demonstrate the existence of peculiar \emph{patterns} on \emph{triods} that are neither \emph{triod twists} nor possess a \emph{block structure} over a \emph{triod twist pattern}, but their \emph{rotation numbers}  are an end point of their respective \emph{forced rotation intervals}, mimicking the behavior of \emph{triod twist patterns}. These \emph{patterns}, absent in circle maps (see \cite{almBB}), highlight a key difference between the rotation theories for \emph{triods} (introduced in \cite{BMR}) and that of circle maps. We name these \emph{patterns}: ``\emph{strangely ordered}" and show that they are semi-conjugate to circle rotations via a piece-wise monotone map. We conclude by providing an algorithm to construct unimodal \emph{strangely ordered patterns} with arbitrary \emph{rotation pairs}.
	\end{abstract}
	
	\maketitle
	
	\section{Introduction}\label{introduction}
	In  the theory of discrete dynamical systems,  periodic orbits (also called cycles) display the simplest type of limit behavior of points.  Understanding how the existence of one periodic orbit affects the existence of another is a central problem here. A notable solution to this problem was obtained when in 1964, A.N. Sharkovsky formulated the celebrated Sharkovsky Theorem (\cite{shatr}), which furnishes a complete description of all possible sets of periods of periodic orbits of a given continuous interval map. To understand this theorem,  it is necessary to introduce a special ordering of the set $\mathbb{N}$  of natural numbers called the Sharkovsky ordering: 

$$3\sha 5\sha 7\sha\dots\sha 2\cdot3\sha 2\cdot5\sha 2\cdot7 \sha \dots $$
$$
\sha\dots 2^2\cdot3\sha 2^2\cdot5\sha 2^2\cdot7\sha\dots\sha 8\sha
4\sha 2\sha 1$$

We denote by $Sh(k)$,  the set of all natural numbers $m$ such that
$k\sha m$, together with $k$, and by $Sh(2^\infty)$ the set
$\{1,2,4,8,\dots\}$ which consists of all powers of $2$. For a continuous map $f$, let us denote the set of periods of its cycles by $Per(f)$. Then, Sharkovsky Theorem can be stated as follows: 

\begin{theorem}[\cite{shatr}]\label{t:shar}
	If $f:[0,1]\to [0,1]$ is a continuous map, $m\sha n$ and $m\in
	Per(f)$, then $n\in Per(f)$. Therefore, there exists $k \in  \mathbb{N}$ $
	\cup$  $ \{2^\infty\}$ such that $Per(f)=Sh(k)$. Conversely, if $k\in
	\mathbb{N} \cup \{2^\infty\}$ then there exists a continuous map
	$f:[0,1]\to [0,1]$ such that $Per(f)=Sh(k)$. 		
\end{theorem}  

 Theorem \ref{t:shar} offers an elegant criterion for determining whether the existence of a cycle of a given period in a continuous interval map guarantees the existence of a cycle with a different period. However, period alone provides a coarse way to describe a cycle. Consequently, researchers have explored analogous coexistence rules based on finer properties of cycles that are also invariant under topological conjugacy. Over time, these efforts have broadened to include more general types of orbits, such as homoclinic trajectories, and have been extended to encompass more complex spaces beyond the real line. For a comprehensive overview of these advancements, see the books \cite{alm00} and \cite{BOM}, along with the paper \cite{BME}. The latter mentioned approach of extending Sharkovsky Theorem to a more general space was taken up for maps of \emph{triods} in \cite{alm98}, \cite{Ba}, \cite{BMR}, and \cite{BB5}.
 
 A \emph{triod} $T$ is defined as the set of complex numbers $z \in \mathbb{C}$ such that $z^3 \in [0,1]$. Essentially, a \emph{triod} consists of a central \emph{branching point} $a$ with three copies of the unit interval attached at their end-points. Each component of $T - \{a\}$ is called a \emph{branch} of $T$. The structure of the set of periods of periodic orbits for a map $f: T \to T$ with the central \emph{branching point} $a$ fixed was explored in \cite{alm98} and \cite{Ba}. A detailed description of these findings was presented in \cite{almnew}, where it was shown that the set of periods of such a map can be expressed as unions of ``initial segments" of the linear orderings associated with all rationals in the interval $(0,1)$, with denominators of at most $3$, defined on specific subsets of rational numbers.  However, this result was only observed, not explained.   In \cite{BMR}, Blokh and Misiurewicz introduced  rotation theory for \emph{triods}, providing a comprehensive explanation of the observed phenomenon.

In \cite{BB5}, the simplest patterns of \emph{triods} with a given \emph{rotation number}: \emph{patterns} which don't force other \emph{patterns} with the same \emph{rotation number}, called \emph{triod twists}, were investigated, and a complete characterization of such \emph{patterns} was obtained. Further,  the description of the dynamics of all possible unimodal \emph{triod twist patterns} with a given \emph{rotation number} was also formulated.

 In this paper, we continue study of maps on \emph{triods}, building upon the works in \cite{BB5} and \cite{BMR} and prove some interesting results. We prove a sufficient condition for a \emph{pattern} $\pi$ on a \emph{triod} $T$ to have its \emph{rotation number} $\rho_{\pi}$ correspond to an endpoint of the \emph{rotation interval} $I_{\pi}$ \emph{forced} by the \emph{pattern} $\pi$.  Leveraging this result, we demonstrate the existence of peculiar \emph{patterns} on \emph{triods}, termed \emph{strangely ordered}, which are neither \emph{triod twists} nor has a \emph{block structure} over a \emph{triod twist pattern}, but   possess \emph{rotation numbers} equal to an endpoint of their respective \emph{forced rotation intervals}. This discovery marks a significant departure from the previously observed phenomenon (see \cite{BMR}),  according to which rotation theory for circle maps (developed in \cite{mis82} and \cite{mis89} by Misiurewicz) and rotation theory for maps on \emph{triods} are analogous.  Notably in \cite{almBB}, it was shown that for \emph{patterns} on circle lacking a \emph{block structure} over \emph{twist patterns} (counter part of \emph{triod twist patterns} for circle maps), the \emph{rotation number} of the \emph{pattern} must lie in the interior of the \emph{rotation interval forced} by the \emph{pattern}. In contrast, our findings highlight an interesting point of deviation between the two theories.

Subsequently, we show that given a \emph{strangely ordered periodic orbit} $P$ with \emph{rotation number} $\frac{p}{q}$, g.c.d$(p,q)=1$ and \emph{rotation pair} $(kp,kq)$ where $k \in \mathbb{N}$, there exists a piece-wise monotone map $ \Phi : T \to \uc$ which semi-conjugates $P$ to circle rotation on $\uc$ by angle $\frac{p}{q}$.

To conclude, we show that \emph{strangely ordered periodic orbits} exists with arbitrary \emph{rotation pairs}. Given, $\frac{p}{q} \in \mathbb{Q}$, g.c.d$(p,q)=1$ and $k \in \mathbb{N}$, we substantiate an algorithm to construct unimodal \emph{strangely ordered patterns} with \emph{rotation pair} $(kp,kq)$.  The  paper is organized into three sections as follows:
 
 \begin{enumerate}
 	\item  Section 1: ``Introduction", provides an overview of the scope and objectives of the paper.
 	
 	\item Section 2: ``Preliminaries", reviews the foundational concepts and key results from rotation theory for \emph{triods}, as developed in \cite{BMR} and \cite{BB5}, which are essential for presenting the main results of this work.
 	
 	\item  In Section 3: ``Main Section", is the main section of the paper where we prove our results. 

 \end{enumerate}

 \section{Preliminaries}\label{preliminaries}

 \subsection{Notations} We denote by $\mathcal{U}$, the set of all continuous maps $f: T \to T$ for which the central \emph{branching point} $a$ of $T$ is the unique fixed point. For $x,y \in T$, we write $ x>y$,  if $x$ and $y$ lie on the same \emph{branch} of $T$ and $x$ is farther away from $a$ than $y$; write $ x \geqslant y$ if $x>y$ or $x=y$. Also, for any set $A$, let $[A]$ denote its convex hull.

 \subsection{Patterns} We call two cycles $P$ and $Q$ on the \emph{triod} $T$ \emph{equivalent} if there exists a homeomorphism $h: [P] \to [Q]$ conjugating $P$ and $Q$ and fixing \emph{branches} of $T$. The  equivalence classes of conjugacy of a cycle $P$ is called the \emph{pattern} of $P$. A cycle $P$ of a map $f \in \mathcal{U}$ is said to \emph{exhibit} a \emph{pattern} $A$ if $P$ belongs to the equivalence class $A$.    We call a cycle and its \emph{pattern} \emph{primitive} if each of its points lies on a different \emph{branch} of $T$. 
 
 A \emph{pattern} $A$ \emph{forces} a \emph{pattern} $B$ if and only if any map $f \in \mathcal{U}$ with a cycle of \emph{pattern} $A$ also has a cycle of \emph{pattern} $B$. It follows (see \cite{alm98,BMR}) that if a \emph{pattern} $A$ \emph{forces} a \emph{pattern} $B \neq A$, then $B$ doesn't \emph{force} $A$.  We say that a cycle $P$ \emph{forces} a cycle $Q$ if the \emph{pattern} \emph{exhibited} by $P$ \emph{forces} the \emph{pattern} exhibited  by $Q$. There is a beautiful way to find all \emph{patterns} \emph{forced} by a \emph{pattern} $A$.

 A map $f \in \mathcal{U} $ is called $P$-\emph{linear} for a cycle $P$,  if it fixes $a$, is \emph{affine} on every component of $[P] - (P \cup \{ a\})$ and also constant on every component of $T - [P]$ where $[P]$ is the convex hull of $P$.

 \begin{theorem}[\cite{alm98,BMR}]\label{forcing}
 	Let $f$ be a $P$-linear map where $P$ is a cycle of pattern $A$. Then a pattern $B$ is forced by $A$ if and only if $f$ has a cycle $Q$ of pattern $B$. 
 \end{theorem}

 For a cycle $P$, call a map $f$, $P$-\emph{adjusted} if $f$ has no cycle other than $P$  with the same \emph{pattern} as $P$.

 \begin{theorem}[\cite{alm98}]\label{P:adjusted}
 	For a cycle $P$ of  a map $f \in \mathcal{U}$, there exists a $P$-adjusted  map $g $ which agrees with  $f$ on $P$, that is, $f |_P = g |_P$. 
 \end{theorem}

 \subsection{Loops of Intervals}(\cite{alm00, alm98})\label{loops} 
 For $x,y \in T$ lying in the same \emph{branch}, we call the convex hull $[x,y]$ of $x$ and $y$, an \emph{interval} on $T$ connecting $x$ and $y$. Let $P$ be a cycle on $T$ and let $f$ be the $P$-\emph{linear} map. Each component of $[P] - (P \cup \{ a\})$ is called a $P$-\emph{basic \emph{interval}} on $T$.
 
  We say that an \emph{interval} $J$ on $T$ $f$-\emph{covers} an \emph{interval} $K$ on $T$ if $f(J) \supseteq K$. Let $\{I_n\}$, $n\in \mathbb{N}$ be a sequence of \emph{interval}s such that $I_j$ $f$-\emph{covers} $I_{j+1}$,  for all $ j \in \mathbb{N} $; then we say that $\{I_n\}$, $n\in \mathbb{N}$ is an \emph{$f-$chain} of \emph{intervals} on $T$.  If a finite $f$-\emph{chain} $\alpha = \{I_0, I_1, \dots , I_{k-1} \}$ of \emph{intervals} on $T$ is such that $I_{k-1}$ also $f$-\emph{covers} $I_0 $, then we call $\alpha = \{I_0, I_1, \dots , I_{k-1} \}$, an
 $f$-\emph{loop} of \emph{intervals} on $T$. Call an \emph{interval} on $T$ \emph{admissible}  if one of its end-points is $a$. We  call a  \emph{chain} (a \emph{loop}) of \emph{admissible} \emph{intervals}  on $T$,   an \emph{admissible loop} (\emph{chain}) on $T$  respectively.  For a cycle $P$ and any $x \in P$, the \emph{loop} of \emph{intervals} $ \Gamma : [x,a]\to [f(x), a] \to [f^2(x), a] \to \dots [f^{n-1}(x), a] \to [x,a], x \in P$ is called the \emph{fundamental admissible loop} of \emph{intervals} \emph{associated} with the cycle $P$. The following result is well-known:

 \begin{theorem}[\cite{alm98, zie95}]\label{theorem:interval:graph}
 	
 	For a loop of interval	$ \Gamma: I_0 \to I_1 \to $ $    \dots $ $  I_{q-1}  \to I_0 $ of length $q$ on $T$,  there exists a point $x_0 \in I_0$ satisfying $f^i(x_0) \in I_{i} $ for $ i \in \{0, 1, 2, $ $ \dots q-1\}$ and $f^q(x_0) = x_0$. 
 \end{theorem}
 
 \begin{remark}\label{loop:associated:point}
 	We call such a point $x_0$, the point \emph{associated} with the \emph{loop} of 
 	\emph{intervals}, $\Gamma$. 
 \end{remark}

 \subsection{Rotation Theory for triods}
 Now, we are in a position to state the rotation theory for \emph{triods} as introduced in \cite{BMR}. In our model of  $T$  we consider $T$ as being embedded into the plane with the central \emph{branching} point at the origin and \emph{branches} being segments of straight-lines. Let us name the \emph{branches} of $T$ in the anticlockwise direction such that $B = \{ b_i | i = 0,1,2 \}$  (addition in the subscript of $b$  is modulo 3) is the collection of all its \emph{branches}.
 
 Let $ f \in \mathcal{U}$ and  $P \subset T- \{a \}$  be finite. By an \emph{oriented graph} corresponding to $P$,  we shall mean a graph $G_P$ whose vertices are elements of $P$ and arrows are defined as follows. For $x,y \in P$, we will say that there is an arrow from $x$ to $y$ and write $x \to y$ if there exists $ z \in T$ such that $ x \geqslant z$ and $ f(z) \geqslant y$. We will refer to a \emph{loop}  in the \emph{oriented graph} $G_P$ as a \emph{``point" loop} in $T$. We call a \emph{point loop} in $T$ \emph{elementary} if it passes through every vertex of $G_P$ at most once. If $P$ is a cycle of period $n$, then the \emph{loop} $\gamma : x \to f(x) \to f^2(x) \to f^3(x) \to \dots f^{n-1}(x) \to x$, $x \in P$ is called the \emph{fundamental point loop} associated with $P$

 Let us assume that the \emph{oriented graph} $G_P$ corresponding to $P$ is transitive (that is,  there is a path from every vertex to every vertex). Observe that, if $P$ is a cycle, then $G_P$ is always transitive.  Let $A$ be the set of all arrows of the \emph{oriented graph} $G_P$. We define a \emph{displacement function} $ d : A \to \mathbb{R}$  by $d(u \to v) = \frac{k}{3}$ where $ u \in b_i$ and $v \in b_j$ and $ j = i +k $ (modulo 3).  For a \emph{point} \emph{loop} $\Gamma$ in $G_P$ denote by $ d(\Gamma)$ the sum of the values of the \emph{displacement} $d$ along the \emph{loop}. In our model of $T$, this number tells us how many times we revolved around the origin in the anticlockwise sense. Thus, $ d(\Gamma)$ is an integer. We call  $ rp(\Gamma) = (d(\Gamma), |\Gamma|)$  and $ \rho(\Gamma) = \frac{d(\Gamma)}{ |\Gamma|}$, the \emph{rotation pair} and \emph{rotation number} of the \emph{point loop}, $\Gamma$  respectively (where $|\Gamma|$ denotes length of $\Gamma$).  
 
 The closure of the set of \emph{rotation numbers} of all \emph{loops} of the \emph{oriented graph},  $G_P$ is called the \emph{rotation set} of $G_P$ and denoted by $L(G_P)$. By \cite{zie95}, $L(G_P)$ is equal to the smallest interval containing the \emph{rotation numbers} of all \emph{elementary loops} of $G_P$. \emph{Rotation pairs} can be represented in an interesting way using the notations used in  \cite{BMR}. We  represent the \emph{rotation pair} $rp(\Gamma) = (mp, mq)$ of a \emph{point loop} $\Gamma$ in $G_P$ with $p,q$ co-prime and $m$ being some natural number as a pair $ (t, m)$ where $ t = \frac{p}{q}$. We call the latter pair, the \emph{modified-rotation pair(mrp)} of the \emph{point loop} $\Gamma$ and write $mrp(\Gamma) = (t, m)$.

 The  \emph{rotation number}, \emph{rotation pair} and \emph{modified rotation pair} of a cycle $P$ are defined as those of  the \emph{fundamental point loop} $\Gamma_P$, corresponding to $P$. Similarly, the \emph{rotation number, rotation pair} and \emph{modified rotation pair} of a \emph{pattern} $A$ are defined as those of any cycle $P$, which \emph{exhibits} $A$. The \emph{rotation interval} \emph{forced} by a \emph{pattern} $A$ is defined to be the \emph{rotation set} $L(G_P)$ of the \emph{oriented graph} $G_P$ corresponding to the \emph{cycle} $P$, where $P$ exhibits $A$. The set of \emph{modified rotation pairs} of all \emph{patterns forced} by a \emph{pattern} $A$ is denoted by $mrp(A)$.

 \emph{Modified rotation pairs} can be visualized as follows. Think of the real line with a prong attached at each rational point and the set $\mathbb{N}
 \cup \{2^\infty\} $ marked on this prong in the Sharkovsky ordering $ \sha$ with $1$ closest to the real line and $3$ farthest from it. All points of the real line are marked $0$; at irrational points we can think of degenerate prongs with only $0$ on them. The union of all prongs and the real line is denoted by $\mathbb{M}$. Then,  a \emph{modified rotation pair} $(t,m)$ corresponds to the specific element of $ \mathbb{M}$, namely to the number $m$ on the prong attached at $t$.  However, no \emph{rotation pair} corresponds to $(t,2^{\infty})$ or to $(t,0)$. Then, for  $(t_1, m_1) $ and  $  (t_2,m_2)$, in $\mathbb{M}$, the convex hull $[(t_1, m_1), $ $  (t_2,m_2)]$  consists of all \emph{modified rotation pairs} $(t,m)$ with $t$ strictly between $t_1$ and $t_2$ or $t=t_i$ and $m \in Sh(m_i)$ for $ i=1,2$.

 A \emph{pattern} $A$ is called \emph{regular} if $A$ doesn't \emph{force} a \emph{primitive pattern} of period $2$; call a cycle $P$ \emph{regular} if it exhibits a \emph{regular pattern}.  By transitivity of \emph{forcing}, \emph{patterns forced} by a \emph{regular pattern} are \emph{regular}. A map $f \in \mathcal{U}$ will be called \emph{regular} if all its cycles are \emph{regular}.

 \begin{theorem}[\cite{BM2}]\label{result:1}

 	Let $A$ be a regular pattern for a map $f \in \mathcal{U}$. Then there are patterns $B$ and $C$ with modified rotation pairs $(t_1, m_1)$ and $(t_2,m_2)$ respectively such that $mrp(A) = [(t_1, m_1), (t_2, m_2)]$. 
 \end{theorem}

 \subsection{Triod-twist patterns}\label{section:triod:twist:color}
 
A \emph{regular pattern} $A$ is called a \emph{triod twist pattern} if it doesn't \emph{force} another \emph{pattern} with the same \emph{rotation number}.  To describe the dynamics of \emph{triod-twist patterns} we need to  use the following \emph{color} notations from \cite{BMR}. An arrow $ u \to v$ in $G_P$ where $u,v \in P$ will be called  \emph{green}, \emph{black} or \emph{red} according as $d(u \to v) =  0,  \frac{1}{3}$ or $\frac{2}{3}$ respectively.  If $P$ is a cycle of $f \in \mathcal{U}$, then we define the \emph{color} of a point $x \in P$ to be the \emph{color} of the arrow $ x \to f(x)$ in its \emph{fundamental point loop}. For an ordering  $\{ b_i | i = 0,1,2 \}$  (addition in the subscript of $b$  is modulo 3) of the \emph{branches} of the \emph{triod} and for a \emph{regular cycle} $P$, we denote by $p_i$, the point of  $P$ closest to the \emph{branching point} $a$ in each \emph{branch} $b_i$.

\begin{theorem}[\cite{BB5}]\label{canonical:numberinh}
For any regular cycle $P$, there exists a  ordering  $\{b_i | i = 0,1,2\}$  (addition in the subscript of $b$  is modulo 3) of the branches of the triod such that $p_i,$  $i = 0,1,2$ are all black. We call this ordering, the canonical ordering of the branches of $T$. 
\end{theorem}
	 From now on throughout the rest of the paper, we will assume that the \emph{branches} has been \emph{canonically ordered}. Then, the subsequent result follows. 

\begin{theorem}[\cite{BB5}]\label{rot:one:third}
	The pattern $\pi_3$, associated with primitive cycle of period $3$  is the unique triod twist pattern with rotation number $\frac{1}{3}$. 
\end{theorem}

In \cite{BB5}, a bifurcation in the qualitative nature of a \emph{triod-twist pattern} (with respect to \emph{color} of points) at \emph{rotation number} $\frac{1}{3}$ was obtained. 

\begin{theorem}[\cite{BB5}]\label{bifurcation:one:third}
	Let $A$ be a triod twist pattern of rotation number $\rho$. Then, if $\rho <  \frac{1}{3}$, $A$ has no red points and if $\rho >  \frac{1}{3}$, $A$ has  no green points.

\end{theorem}
We now state a  necessary condition for a given \emph{pattern} $\pi$ to be a \emph{triod-twist pattern}. 
\begin{definition}[\cite{BB5}]\label{green}
	 A \emph{regular cycle} $P$ is called \emph{green} if for any two points $x,y \in P$, $ x > y $ such that $f(x)$ and $f(y)$ lie in the same \emph{branch} of $T$, we have $ f(x) >  f(y)$. Call a \emph{pattern} $A$ \emph{green},  if any cycle which exhibits it,  is \emph{green}. 	
\end{definition}
 
 \begin{theorem}[\cite{BB5}]\label{necesary:condition:triod:twist}
 	A triod-twist pattern is green. 
 \end{theorem}

 We now introduce a special tool which gives us a complete characterization of a \emph{triod-twist pattern}. 
 
 \begin{definition}[\cite{BB5}]\label{code:function}
 	 The \emph{code function} for a cycle $P$  of a map $ f \in \mathcal{U}$  with \emph{rotation number} $\rho \neq \frac{1}{3} $   is a function $ \psi: P \to \mathbb{R}$    defined as follows: 
 	
 	\begin{enumerate}
 		\item choose any point $x_0 \in P$ and set $ \psi(x_0) =0$. 
 		
 		\item define $\psi$ iteratively on other points of $P$  as follows:  
 		
 		\raggedright for each $ k \in \mathbb{N} $ set $ \psi \left ( f^k(x_0) \right ) = \psi(x_0) + k \rho - [t_k]$ where  $ t_k = \displaystyle \sum_{j=0}^{k-1} d \left (f^j(x_0), f^{j+1}(x_0) \right )$ and for $x \in \mathbb{R}$, $[x]$ denotes the greatest integer less than or equal to $x$.

 	\end{enumerate}
 	
 \end{definition}

 	For any $ x \in P$, the value $\psi(x)$ is called the \emph{code} of the point $x$.  If $P$ is a cycle of period $q$,  then $t_q = q \rho \in  \mathbb Z_+$ and hence $ \psi(f^q(x_0)) = \psi(x_0)$. So the definition is consistent.  It is easy to see that  $ \psi(x)- \psi(y)$ is independent of the choice of the point $x_0$ for $x,y \in P$. 
 	
 	\begin{definition}[\cite{BB5}]\label{non:decreasing}
 		
 		The \emph{code function} $\psi$  of a cycle $P$ with \emph{rotation number} $\rho \neq \frac{1}{3}$  is said to be \emph{non-decreasing} if the following holds:
 		
 		\begin{enumerate}
 			\item if $ \rho < \frac{1}{3}$, then  $\psi(x) \leqslant \psi(y )$ whenever $ x> y$, and 
 			
 			\item if $ \rho > \frac{1}{3}$, then $\psi(x) \geqslant \psi(y )$ whenever $ x> y$.
 		\end{enumerate}

Otherwise, the \emph{code function} $\psi$ is referred to as \emph{decreasing}. Additionally, \emph{code function} $\psi$  is called  \emph{strictly increasing} if $\psi$ is \emph{non-decreasing} and the \emph{codes} of no two consecutive points of $P$ lying in the same \emph{branch} of $T$ are the same.
 		
 	\end{definition}
 	
 		 A \emph{pattern} $A$ is said to possess a \emph{non-decreasing} or \emph{strictly increasing} \emph{code function} depending on whether any cycle $P$,  \emph{exhibiting} it has a \emph{non-decreasing} or \emph{strictly increasing code function}, respectively.
 	  The following result from \cite{BB5} provides a necessary and sufficient condition for a \emph{pattern} $A$ to qualify as a \emph{triod-twist pattern}.

 \begin{theorem}[\cite{BB5}]\label{tri-od:rot:twist:order:inv}
 	The necessary and sufficient condition for a regular pattern $\pi$  to be a  triod twist pattern  is that it has a strictly increasing code function. 
 \end{theorem}

 \section{Main Section}

We begin by exploring the connection between position of the \emph{rotation number} $\rho(\pi)$ of a \emph{pattern} $\pi$ within the \emph{rotation interval} $I_{\pi}$ \emph{forced} by the \emph{pattern} $\pi$ and the nature of the \emph{code function} \emph{associated} with the \emph{pattern} $\pi$. 

\begin{definition}
A \emph{pattern} $\pi$ is said an \emph{interior pattern} if its \emph{rotation number} $\rho(\pi)$ is an interior point of the \emph{rotation interval} $I_{\pi}$ \emph{forced} by the \emph{pattern} $\pi$. It is called a \emph{frontier pattern} if its \emph{rotation number} $\rho(\pi)$ is an end point of the \emph{rotation interval} $I_{\pi}$ \emph{forced} by the \emph{pattern} $\pi$. 
\end{definition}

 \begin{theorem}\label{decreasing:conclusion}
 	A pattern $\pi$ on a triod $T$ with a  decreasing code is an interior pattern. 
 	\end{theorem}
 
 \begin{proof}
 	We consider	 the case $\rho(\pi) = \rho \leqslant \frac{1}{3}$, the case when $\rho(\pi) > \frac{1}{3}$ is similar. Let $P$ be a periodic orbit of period $q$ which \emph{exhibits} $\pi$ and let $f$ be a $P$-\emph{linear} map. Since, $\pi$ has a \emph{decreasing code}, there exists $x,y \in P$ such that  $x> y$ and $\psi(x) > \psi(y)$. Let $x= f^n(y)$ for $n \in \mathbb{N}$, $n<q$. Consider the finite sequence of intervals $I_i$, $i \in \{0,1,2, \dots n\}$ such that $I_0 = [y,a], I_1 = [f(y), a], \dots, I_n = [f^n(y), a] = [x,a] \supseteq [y,a]$. Clearly, $I_i \supseteq I_{i+1}$ for all $i \in \{ 0,1, 2, \dots n-1\}$ and hence $I_0 \to I_1 \to I_2 \to \dots I_{n-1} \to I_n \to I_0$ is an \emph{admissible loop} of intervals. So, by Theorem \ref{theorem:interval:graph}, there exists a periodic point $z \in I_0$ such that $f^i(z) \in I_i$ for $i \in \{0,1, \dots n-1\}$ and $f^n(z) = z$. The period of $z$ is either $n$ itself or a divisor of $n$; the \emph{rotation number} of $z$ is $\frac{m}{n}$ for some $m \in \mathbb{N}$. The \emph{code} for the point $x$ is $\psi(x) = \psi(f^n(y)) = \psi(y) + n \rho- m$. Now, from $\psi(x) > \psi(y)$, we get, $\psi(y) + n\rho - m > \psi(y)$ which yields $\frac{m}{n} < \rho \leqslant \frac{1}{3}$. Thus, $f$ has a periodic orbit $Q$ with \emph{rotation number} $\frac{m}{n} < \rho$ and hence by Theorem \ref{forcing}, $\pi$ \emph{forces} a \emph{pattern} $\pi'$ with \emph{rotation number} $\rho(\pi') < \rho(\pi) \leqslant \frac{1}{3}$. The result now follows from Theorem \ref{result:1}. 
 \end{proof}

 \begin{theorem}\label{equal:code:non:coprime}
If for a pattern $\pi$ on a triod $T$, the codes of two points of $T$ lying in the same branch are equal, then $\pi$ has a non-coprime rotation pair. 
\end{theorem}

\begin{proof}
	 Let $P$ be a periodic orbit of period $q$, \emph{rotation number} $\frac{p}{q}$ which \emph{exhibits} $\pi$ and let $f$ be a $P$-\emph{linear} map. It is sufficient to consider the case, $\frac{p}{q} \leqslant \frac{1}{3}$, the arguments in the other case are analogous.  Suppose, there exists $u,v \in P$ such that  $u> v$ and $\psi(u) = \psi(v)$. Now, there exists $s \in \mathbb{N}$, $s<q$ such that $u= f^s(v)$. Consider the sequence of intervals $I_i$, $i \in \{0,1,2, \dots s\}$ such that $I_0 = [v,a], I_1 = [f(v), a], \dots I_s = [f^s(v), a] = [u,a] \supseteq [v,a]$. Clearly, $I_i \supseteq I_{i+1}$ for all $i \in \{ 0,1, 2, \dots s-1\}$ and hence $I_0 \to I_1 \to I_2 \to \dots I_{s-1} \to I_s \to I_0$ is an \emph{admissible loop} of intervals. So, by Theorem \ref{theorem:interval:graph}, there exists a periodic point $w \in I_0$ such that $f^i(w) \in I_i$ for $i \in \{0,1, \dots s-1\}$ and $f^s(w) = w$. The period of $w$ is either $s$ itself or a divisor of $s$; the \emph{rotation number} of $w$ is $\frac{r}{s}$ for some $r \in \mathbb{N}$. The \emph{code} for the point $u$ is $\psi(u) = \psi(f^s(v)) = \psi(v) + s \frac{p}{q}- r$. Now, from $\psi(u) = \psi(v)$, we get, $\psi(v) + s\frac{p}{q} - r = \psi(v)$ which yields $\frac{r}{s} = \frac{p}{q}$. But, $s< q$, this means $p$ and $q$ cannot be co-prime. 
\end{proof}

\begin{theorem}\label{non:twist:non:coprime}
	If a pattern $\pi$ on a triod $T$ with coprime rotation pair is not a triod twist pattern, then it is an interior pattern. 

\end{theorem}

\begin{proof}
	Since, $\pi$ is not a \emph{triod-twist pattern}, by Theorem \ref{tri-od:rot:twist:order:inv}, the \emph{code} for $\pi$ cannot be \emph{strictly increasing}. Also, since the \emph{rotation pair} of $\pi$ is co-prime, it follows from Theorem \ref{equal:code:non:coprime}, that the \emph{code} for $\pi$ is \emph{decreasing}. Now, the result follows from Theorem \ref{decreasing:conclusion}. 

\end{proof}	

Theorem \ref{non:twist:non:coprime} gives us a necessary and sufficient condition for a \emph{pattern} $\pi$ on a \emph{triod} $T$ to be a \emph{triod-twist pattern} based upon its \emph{forced rotation interval}. 

\begin{theorem}
	The necessary and sufficient condition for a pattern $\pi$ on a triod $T$ to be a triod twist pattern is that it is a frontier pattern with coprime rotation pair. 
\end{theorem}

\begin{proof}
	The necessary part is trivial and follows directly from definition of \emph{triod-twist pattern} and Theorem \ref{result:1}. To prove the sufficient part, let $\pi$ be a \emph{frontier pattern} with \emph{coprime rotation pair}. If $\pi$ is not a \emph{triod twist pattern}, then by Theorem \ref{non:twist:non:coprime}, $\pi$ is an \emph{interior pattern}, a contradiction and hence the result follows. 
\end{proof}

Thus, \emph{triod-twist patterns} are the ``\emph{canonical frontier patterns}". What other \emph{frontier patterns} does a continuous map on a \emph{triod} posses? To answer this question, we introduce the notion of \emph{block structure} for \emph{cycles} of \emph{triods} similar to one introduced for interval maps by Blokh and Misiurewicz (see \cite{alm00}). 

\begin{definition}
	Let $P$ be a cycle on a \emph{triod} $T$ and let $f$ be a $P$-\emph{linear} map. We say that $P$ has a \emph{block structure} over a cycle $Q$ if $P$ can be divided into subsets $P_1, P_2, \dots P_m$,  called \emph{blocks} of the same cardinality, where $m$ is the period of $Q$, the sets $[P_i]$ are pairwise disjoint, none of them contains the \emph{branching point} $a$ of $T$, each of them contains one point $x_i$ of $Q$ and $f(P_i) = P_j$ whenever $f(x_i) = x_j$. 
\end{definition}

We use the same terminology for \emph{patterns}. In particular, a \emph{pattern} $A$ has a \emph{block structure} over a \emph{pattern} $B$ if there exists a cycle $P$ of \emph{pattern} $A$ with a \emph{block structure} over a cycle $Q$ of \emph{pattern} $B$. We now show that if a \emph{pattern} $A$ on a \emph{triod} $T$ has a \emph{block structure} over a \emph{triod-twist pattern} $B$, then $A$ is a \emph{frontier pattern}. For this we first prove the following theorem.

\begin{theorem}\label{non:decreasing:conclusion}
A pattern $\pi$ on a triod $T$ with a non-decreasing code is a frontier pattern. 
\end{theorem}

\begin{proof}
	We first consider the case $\rho(\pi) = \rho \leqslant \frac{1}{3}$; the case when $\rho(\pi) > \frac{1}{3}$ is similar. By way of contradiction, suppose $\rho$ is not an endpoint of the \emph{rotation interval} $I_{\pi}$ \emph{forced} by the \emph{pattern} $\pi$.  Let $I_{\pi}$ be the \emph{rotation interval forced} by the \emph{pattern} $\pi$.  Let $P$ be a \emph{periodic orbit} which \emph{exhibits} $\pi$ and let $f$ be a $P$-\emph{linear map}. Since, $\rho $ is not an end point of $I_{\pi}$, so $f$ has a periodic orbit $Q$ has \emph{rotation number}, $\frac{p}{q} < \rho \leqslant \frac{1}{3}$ where g.c.d$(p,q) =1$.

	Let $x \in Q$. For each $i \in \{ 0,1,2, \dots, q-1\}$, choose the point $y_i$ closest to the \emph{branching point} $a$ in the same \emph{branch} as $f^i(x)$ such that  $f(y_i) = f(f^i(x))$. Then, it is easy to see that $[y_0,a] \to [y_1,a] \to \dots [y_{q-1}, a] \to [y_0,a]$ is an \emph{admissible loop} of \emph{intervals}, so there exists a periodic point $ z\in T$ of period $q$ such that $y_i \geqslant f^i(z)$ for all $i \in \{0,1,2, \dots q-1\}$.  Clearly, the \emph{rotation numbers} of the cycles containing the points $x$ and $z$ are the same.  So, we can assume without loss of generality that the originally chosen point $x$ is $z$.

	Observe that each $t \in Q       $ belongs to some $P$-\emph{basic interval} $[u,v]$ where $u>v$. We define a function $h: Q \to P$ as follows: Define $h(t) =u$. We claim $f(h(t)) > h(f(t))$. Since, $f$ is $P$-\emph{linear} one of $f(u)$ and $f(v)$ must lie in the same \emph{branch} of $T$ as $f(t)$. Let us first consider the case, when $f(v)$ and $f(t)$ lie in the same \emph{branch}. Then, $f(t) > f(v)$, otherwise there exists a point $w$ lying between $v$ and $a$ such that $f(w) = f(t)$,  contradicting the choice of $Q$. Since, $f$ is $P$-linear, $f(t) \in [f(u), f(v)]$ and hence it follows that,  $f(u)> f(t)> f(v)$. Thus, $f(h(t)) = f(u) \geqslant h(f(t))$. 
	
	Let us now consider the second case, when $f(v)$ and $f(t)$ don't lie in the same \emph{branch}. Since, $f(t) \in [f(u), f(v)]$, it follow that $f(u) > f(t)$. Thus, $f(h(t)) = f(u) > h(f(t))$, in this case also. It follows that for any $t \in Q$, $\alpha_t : [h(t), a] \to [h(f(t)), a] \to \dots [h(f^{q-1}(t), a)] \to [h(t), a] $ is an \emph{admissible loop of intervals}. Further,  since for each $i \in \{0,1,2, \dots q-1\}$,  $f^i(t)$ and $h(f^i(t))$ lie in the same \emph{branch} of $T$, it follows that the \emph{rotation number} of the \emph{loop} $\alpha_t$ is same as that of $Q$, that is, $\frac{p}{q}$. 
	
	Let $v_0^t = h(t)$, $v_1^t = h(f(t))$, $\dots$, $v_{q-1}^t = h(f^{q-1}(t))$. Then, $\alpha_t$ can be written as $[v_0^t, a] \to [v_1^t, a] \to [v_2^t, a] \to \dots [v_{q-1}^t, a] \to [v_0^t, a]$ where $v_i^t \in P$ and  $f(v_i^t) > v_{i+1}^t$. If $f(v_i^t) \neq v_{i+1}^t$ for some $i \in \{0,1,2, \dots q-1\}$, then $v_{i+1}^t = f^j(v_i^t)$ for some $j
	>1$. In this case, we can replace the arrow, $[v_i^t,a] \to [v_{i+1}^t,a]$ with the \emph{block} $[v_i^t,a] \to [f(v_i^t),a] \to \dots [f^{j-1}(v_i^t),a] \to [v_{i+1}^t,a]$. We do this for every $i$ with $v_{i+1}^t \neq v_i^t$. Then, we get a new \emph{admissible loop} $\gamma_t$, which is either the \emph{fundamental loop} of $P$ or its repetition  and so the \emph{rotation number} of $\gamma_t$ must be equal to that of $P$, that is, $\rho$.
	
	By construction, the \emph{rotation number} of $\gamma_t$ is the weighted average of the \emph{rotation number} of $\gamma_t$ and \emph{inserts} (here by \emph{rotation number} of an \emph{insert} $\delta : [f(v_i^t),a] \to \dots [f^{j-1}(v_i^t), a] \to [v_{i+1}^t, a]$) we mean the quantity $\rho_{\delta} = \frac{1}{j-1} [d(f(v_i^t), f^2(v_i^t)) + \dots +  d(f^{j-1}(v_i^t), v_{i+1}^t)])$.  Consider  an \emph{insert} $\delta : [f(v_i^t),a] \to \dots [f^{j-1}(v_i^t), a] \to [v_{i+1}^t, a]$) and let its \emph{rotation number} be $\rho_{\delta}$. Since, the \emph{code} for $\pi$ is \emph{non-decreasing}, so, since, $f(v_i^t) > v_{i+1}^t$ and $\rho(\pi) < \frac{1}{3}$, we have,  $\psi (f(v_i^t)) \leqslant \psi (v_{i+1}^t)$ (see definition \ref{non:decreasing}). From this,  simple computation yields $ \rho_{\delta} \leqslant \rho $. Thus, the \emph{rotation number} $\rho_{\delta}$ of every \emph{insert} $\delta$ is at-most $\rho$. The \emph{rotation number} $\rho$ of $\gamma_t$ is the weighted average of the \emph{rotation number} $\frac{p}{q} < \rho$ of $\alpha_t$ and the \emph{rotation number} of \emph{inserts} each of which is at-most $\rho$, a contradiction that implies the claim.

\end{proof}

\begin{theorem}\label{block:structure:equivalence}
	A pattern $\pi$ on a triod $T$ having a block structure over a triod-twist pattern is a frontier pattern. 
\end{theorem}

\begin{proof}
	Let $P$ be a periodic orbit that exhibits the \emph{pattern} $\pi$, and let $f$ be a $P$-\emph{linear map}. Denote the \emph{blocks} of $P$ by $P_1,P_2, \dots P_k$ such that collapsing each \emph{block} to a single point produces a cycle $Q$ \emph{exhibiting} a \emph{triod-twist pattern} with \emph{rotation pair} $(p,q)$, where $p$ and $q$ are \emph{co-prime}. Assume each \emph{block} $P_i$, $i \in \{1,2,\dots k\}$ of $P$ consists of $k$ points, making the \emph{rotation pair} of $P$ equal to $(kp,kq)$. The \emph{code} for the \emph{cycle} $P$ is \emph{non-decreasing},  because by construction all points of $P$ within a given \emph{block} $P_i$ of $P$ must have the same value of \emph{code} function,  and the \emph{code} of the cycle $Q$ obtained by collapsing each \emph{block} to a point is \emph{strictly increasing}. The conclusion then follows directly from Theorem \ref{non:decreasing:conclusion}.

\end{proof}

Are \emph{triod-twist patterns} and \emph{patterns} having \emph{block structures} over \emph{triod-twist patterns}, the only \emph{frontier patterns}? This question was previously examined by Misiurewicz in \cite{almBB} for circle maps, where it was demonstrated that the answer is negative for circle maps. Specifically, in \cite{almBB} it was shown that for \emph{patterns} of circle maps that do not have a \emph{block structure} over a \emph{twist periodic orbit} (analogue of \emph{triod twist patterns} for circle maps), the \emph{rotation number} of the circle \emph{pattern} necessarily lies in the interior of the \emph{rotation interval forced} by the \emph{pattern}.

In contrast, we now unveil the existence of peculiar \emph{frontier patterns} on \emph{triods} that are neither \emph{triod-twists} nor has a \emph{block structure} over a \emph{triod-twist pattern}. This finding constitutes a notable deviation from the previously observed phenomenon (see \cite{BMR}), according to which rotation theory for circle maps and rotation theory for maps on \emph{triods} are analogous. Consider the following example (See Figure \ref{strange_orbit_existence}). 

\begin{example}\label{example:existence:strange}
	Consider a \emph{pattern} $\pi$ represented  by a periodic orbit $P$ of period $8$ on $T$, consisting of points $x_1, x_2, x_3, x_4$ in the \emph{branch} $b_0$, points $y_1, y_2$ in the \emph{branch} $b_1$ and points $z_1, z_2$ in the \emph{branch} $b_2$ in the direction away from the \emph{branching point} $a$ such that $f(x_1) = y_1, f(y_1) = z_1$, $f(z_1) = x_2$, $f(x_2) = y_2$, $f(y_2) = z_2$, $f(z_2) = x_4$, $f(x_4) = x_3$, $f(x_3) = x_1$ (See Figure \ref{strange_orbit_existence}). Clearly, the \emph{rotation number} of $P$ and hence of $\pi$ is $\rho(\pi) = \frac{1}{4}$. Let us compute \emph{codes} for the points of $P$.  Assuming $\psi(x_4) =0$, the codes for the remaining points of $P$ are given by $\psi(x_1) =\frac{1}{2}, \psi(x_2) = \frac{1}{4}, \psi(x_3) = \frac{1}{4}, \psi(y_1) = \frac{3}{4}, \psi(y_2) = \frac{1}{2}, \psi(z_1) = 1, \psi(z_2) = \frac{3}{4}$. Thus, the \emph{code function} $\psi$ for $\pi$ is \emph{non-decreasing}. So, by Theorem \ref{non:decreasing:conclusion}, $\pi$ is a \emph{frontier pattern}. But, clearly the \emph{pattern} $\pi$ is neither a \emph{triod-twist pattern} nor has a \emph{block structure} over a \emph{triod-twist pattern}. We will call such \emph{strange patterns} $\pi$,  \emph{strangely ordered}.

	\begin{figure}[H]
		\caption{A \emph{strangely ordered periodic orbit} with \emph{rotation number} $\frac{1}{4}$ and \emph{rotation pair} $(2,8)$}
		\centering
		\includegraphics[width=0.8 \textwidth]{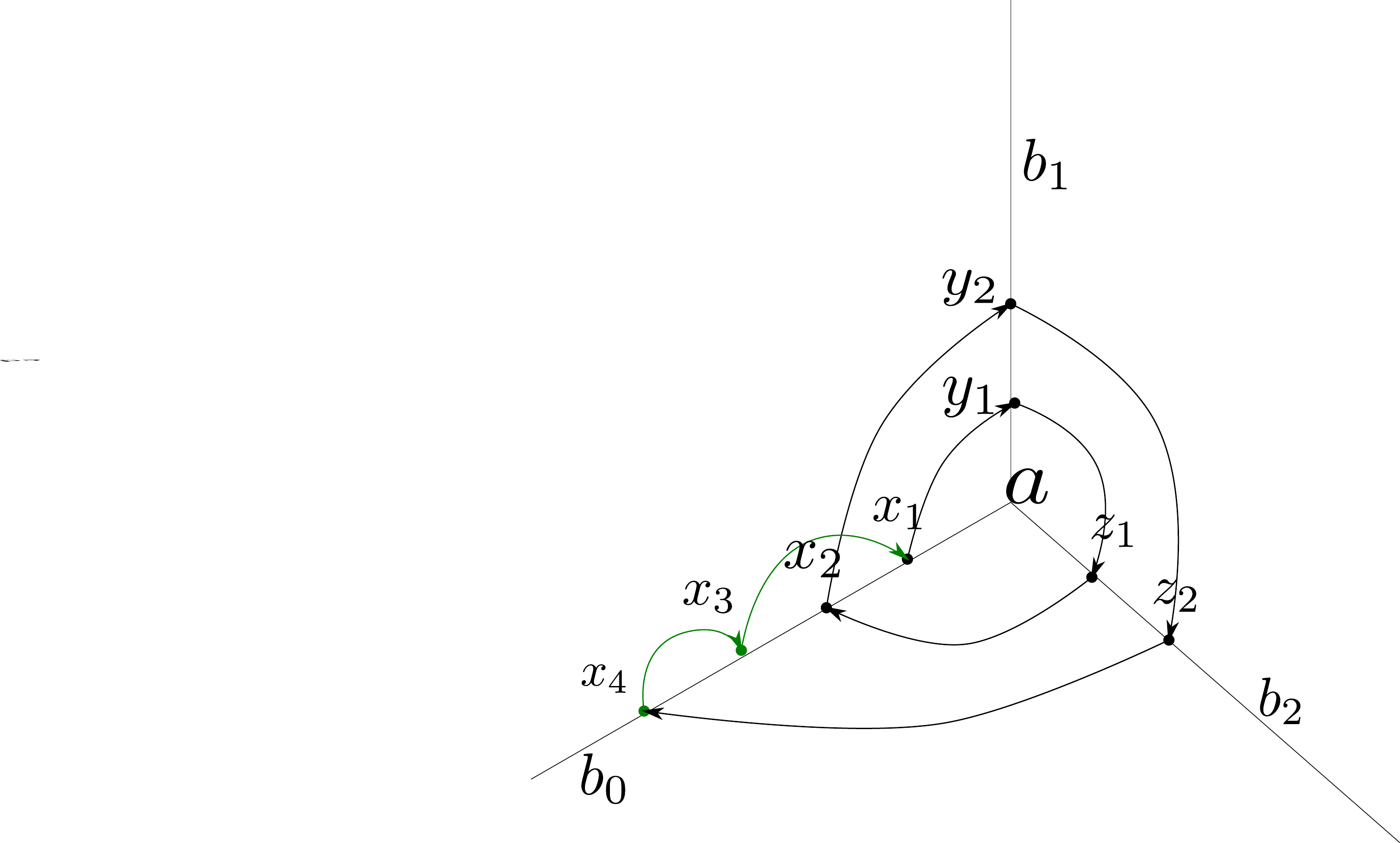}
		\label{strange_orbit_existence}
	\end{figure}
	
\end{example}

\begin{definition}
	A \emph{frontier pattern} $\pi$ is called \emph{strangely ordered} if it  is neither a \emph{triod-twist pattern} nor has a \emph{block structure} over a \emph{triod-twist pattern}.  
\end{definition}

From the above definition and Theorem \ref{decreasing:conclusion}, we have:

\begin{theorem}\label{necessary:strangely:ordered:pattern}
A strangely ordered pattern has a non-decreasing code. 
\end{theorem}

Let $P$ be a \emph{strangely ordered periodic orbit} with \emph{rotation number} $\frac{p}{q}$, g.c.d$(p,q)=1$ and \emph{rotation pair} $(kp,kq)$ where $k \in \mathbb{N}$. Let $f$ be a $P$-\emph{linear} map. We will now show that $P$ is semi-conjugate to circle rotation on $\uc$ by angle $\frac{p}{q}$. 

\begin{definition}\label{coherant:strip}
We define a relation $\sim$ on $P$ as follows. For $x,y \in P$, define,  $x \sim y$ if $x$ and $y$ lie in the same \emph{branch} of $T$ and points of $P$ lying in $[x,y]$ have the same integral value of the \emph{code function} $\psi$. It follows that $\sim$ is an equivalence relation on $P$. We call the convex hulls of each equivalence class under $\sim$, a \emph{coherent-strip} of $P$.
\end{definition}

\begin{theorem}
	There exists a map $\Phi: T \to \uc$ which semi-conjugates $P$ with circle rotation on $\uc$ by angle $\frac{p}{q}$. Further, $\Phi$ is monotone on each coherent-strip of $P$. 
\end{theorem}

\begin{proof}
We parameterize $\uc$ by $[0,1]$ with its end points identified and define $g: [0,1) \to [0,1)$ by $g(x) = x + \frac{p}{q}$ (mod 1). Let $Q$ be the orbit of $0$ for this map. We define $\Phi : P \to Q$ as follows. Pick any arbitrary point $x_0 \in P$ and set $\Phi(x_0) = \psi(x_0) =0$ and then for all $i \in \mathbb{N}$, define $\Phi(f^i(x_0)) = \psi(f^i(x_0))$(mod 1). We have for each $x \in P$, $\Phi(f(x)) = \psi(f(x))$(mod 1) = $\psi(x) + \frac{p}{q}$ (mod 1) = $g(\psi(x)) = g (\Phi(x))$. Thus, $\Phi \circ f = g \circ \Phi$. We extend $\Phi$ from $P$ to $T$ by defining it linearly on each $P$-\emph{basic interval} and such that it is constant on every component of $T- [P]$ where $[P]$ is the convex hull of $P$.
From Theorem \ref{necessary:strangely:ordered:pattern},  \emph{code} for $P$ must be \emph{non-decreasing}. It follows that,  if $A$ is a \emph{coherent-strip} of $P$,  $\psi|_{A}$ is monotone and hence, $\Phi|_{A}$ is monotone. 
\end{proof}

We now show that \emph{strangely ordered patterns} exists for arbitrary rotation pairs. We describe an algorithm to construct uni-modal \emph{strangely ordered patterns} with arbitrary  \emph{rotation pairs}. For this we require the description of unimodal \emph{triod-twist patterns} with a given \emph{rotation pair} obtained in \cite{BB5}. We first introduce  a few notions.  Let us recall \emph{color} notions from Section \ref{section:triod:twist:color}.  Sets of consecutive points of a certain \emph{color} lying in a \emph{branch} is said to form a \emph{block} of that \emph{color}. Thus, we have \emph{green blocks}, \emph{blacks blocks} and \emph{red blocks}.

\begin{figure}[H]
	\caption{Dynamics of $\Gamma_0^{\frac{2}{9}}$}
	\centering
	\includegraphics[width=0.6 \textwidth]{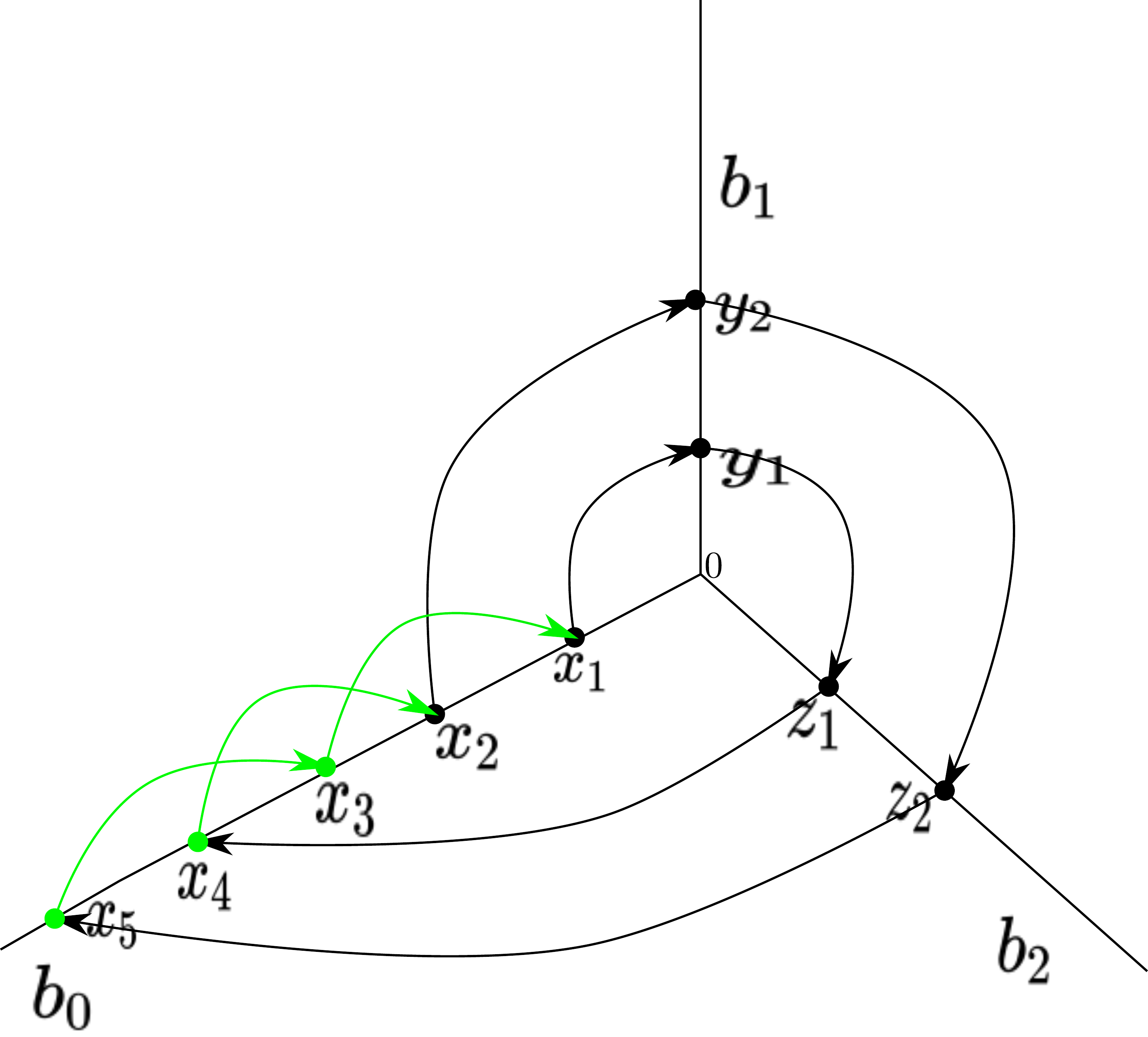}
	\label{rho_less_third_twist_example}
\end{figure}

Let $p$ and $q$ be co-prime natural numbers such that $\frac{p}{q} < \frac{1}{3}$. By \cite{BB5}, there are three distinct unimodal \emph{triod-twist patterns} $\Gamma_{j}^{\frac{p}{q}}, j \in \{0,1,2\}$ with \emph{rotation number} $\frac{p}{q}< \frac{1}{3}$.  Fix some $j \in \{ 0,1,2\}$. Let $P_{j}^{\frac{p}{q}}$ be a periodic orbit which exhibits $\Gamma_{j}^{\frac{p}{q}}$. We assume that the \emph{branches} of $T$ be \emph{canonically ordered} such that the point of $P_{j}^{\frac{p}{q}}$ closest to the \emph{branching point} $a$ in each \emph{branch} of $T$ is \emph{black}. 
Then,  the description of the  dynamics of $P_{j}^{\frac{p}{q}}$ (as obtained in \cite{BB5}) is as follows.  (See Figure \ref{rho_less_third_twist_example}) 

\begin{enumerate}
	\item The \emph{branch} $b_{j}$ consists of $q-2p$ points $x_1, x_2, \dots x_{q-2p}$ numbered in the direction away from $a$. Out of these the points $x_1, x_2 \dots x_p$ are \emph{black} and the points $x_{p+1}, x_{p+2}, \dots x_{q-2p}$ ($q-3p$ in number) are \emph{green}. The \emph{branch} $b_{j +1}$ consists of $p$ points $y_1, y_2, \dots y_p$ numbered in the direction away from $a$, all of which are \emph{black}. The \emph{branch} $b_{j+2}$ contains $p$ points $z_1, z_2, \dots z_p$ in the direction away from $a$, all of which are \emph{black}.
	
	\item $f(x_i) = f(x_{i+p})$ for $i \in \{ p+1, \dots q-2p\}$, that is, the last $q-3p$ \emph{green points} in the \emph{branch} $b_{j}$ (in the direction away from $a$) are shifted to the right by $p$ points.
	
	\item $f(x_i) = y_i$, for $i \in \{1,2, \dots p\}$, that is, the first $p$ points of the \emph{branch} $b_{j}$ (in the direction away from $a$) map to the $p$ points of the \emph{branch} $b_{j+1}$ in an order preserving fashion.
	
	\item $f(y_i) = z_i$, for $i \in \{ 1,2 \dots p \}$, that is, the $p$ points of the \emph{branch} $b_{j+1}$ map to the $p$ points of the \emph{branch} $b_{j+2}$ in an order preserving fashion. 
	
	\item $f(z_i) = x_i$ for $i \in \{1, 2, \dots p\}$, that is, the $p$ points of the \emph{branch} $b_{j+2}$ map to the first $p$ points of the \emph{branch} $b_{j}$ (in the direction away from $a$) in an order preserving fashion. 
\end{enumerate}

We now describe below an algorithm to construct \emph{strangely ordered periodic orbit} with \emph{rotation pair} $(kp,kq), k \in \mathbb{N}$. 

\begin{enumerate}
	\item We take $k$ periodic orbits $P_1, P_2, \dots P_k$, each of which \emph{exhibits pattern} $\Gamma_{j}^{\frac{p}{q}}$. Suppose the orbits are given temporal labeling with critical points $c_{i1}, i \in \{1,2, \dots k\}$ such that $P_i = \{ c_{i1}, c_{i2}, \dots c_{iq}\}$ and $c_{i(j+1)} = f(c_{ij}), j \in \{1,2, \dots q-1\}$. The orbits $P_1, P_2, \dots P_k$ are placed in such a manner that the $P$-\emph{linear map} $f$ where $P = P_1 \cup P_2 \cup \dots P_k$ is uni-modal with critical point $c_{11}$ and for every $i \in \{ 1,2 , \dots q\}$, $c_{1i} > c_{2i} > \dots c_{ki}$. We will call $\mathcal{P}$, with the map $f$ as defined above, the $k$-\emph{tuple lifting} of $\Gamma_{j}^{\frac{p}{q}}$ (See Figure \ref{rho_less_third_doubleton_lifting}). 
	
	\item We now change the map $f$ to form a new map $g$ so that the $k$ cycles are glued into one cycle under $g$ while the map $g$ remains uni-modal and the \emph{code} of the resulting cycle is \emph{non-decreasing}. Observe that from the construction and from the dynamics of $\Gamma_{j}^{\frac{p}{q}}$, it follows that for each $i \in \{1,2,\dots k\}$,  $c_{i4}$ is the point of the cycle $P_i$ farthest away from $a$ in the \emph{branch} $b_j$. Let for each $i \in \{1,2, \dots k\}$,  $c_{i\ell}$ denotes the $p$-th  point of the cycle $P_i$ in the \emph{branch} $b_j$ counting from $c_{i4}$ towards the \emph{branching} point $a$. We set $g(c_{14}) = c_{2\ell}$, that is, the point of $P_1$ farthest away from $a$ in the \emph{branch} $b_j$ is mapped to the $p$-th point of the  cycle $P_2$ in the \emph{branch} $b_j$ counting from the point $c_{24}$ of $P_2$ towards the \emph{branching} point $a$. But, this means that the point $c_{2 (\ell -1)}$ of the cycle $P_2$ that used to be mapped to $c_{2 \ell}$ under $f$ will have to be mapped somewhere else. We set $g(c_{2(\ell-1)}) = c_{3 \ell}$, that is, we move the $f$-image of $c_{2(\ell-1)} $ one step towards the \emph{branching} point $a$. Thus, the point $c_{3(\ell-1)}$ of the cycle $P_3$ that used to be mapped to $c_{3 \ell}$ under $f$ has to be mapped somewhere else. We set $g(c_{3(\ell-1)}) = c_{4 \ell}$. In other words, each point $c_{i(\ell -1)}$ of the cycle $P_i$, $i=2,3, \dots k$. that used to be mapped to $c_{i \ell}$ under $f$ has its image shifted one step towards the \emph{branching} point $a$, and thus, will be mapped  by $g$ to the point $c_{(i+1) \ell}$ which is one step towards the \emph{branching} point $a$  from its old image $c_{i \ell}$.  In particular, on the last step in the construction, the point $c_{k(\ell-1)}$ that used to be mapped to $c_{k \ell}$ under $f$ must have its image shifted one place towards the \emph{branching} point $a$, which brings the $g$-image of this point to $c_{15}$, the point of $\mathcal{P}$ located immediately next to point $c_{k \ell}$ towards the \emph{branching} point $a$ and coinciding with the $f$-image of $c_{14}$. On other points of $\mathcal{P}$, the maps $f$ and $g$ are the same. In this way, the $f$-periodic orbits $P_1, P_2, \dots P_k$ gets merged into one $g$-orbit $\mathcal{P}$ (See Figure \ref{rho_less_third_strangely_ordered}).  
	
	\item A simple computation shows that the \emph{code} for $g$ is \emph{non-decreasing}.  Also, clearly by construction,  $\mathcal{P}$ is neither a \emph{twist pattern} nor has a \emph{block structure} over a \emph{twist pattern}. Hence, by Theorem \ref{non:decreasing:conclusion},  $\mathcal{P}$ is a \emph{strangely ordered periodic orbit} of \emph{rotation number} $\frac{p}{q} < \frac{1}{3}$ and \emph{rotation pair} $(kp,kq)$. Finally, observe that we can choose any $j \in \{0,1,2\}$, so the algorithm yields three distinct \emph{strangely ordered patterns} with \emph{rotation pair} $(kp,kq)$.

\end{enumerate}

	\begin{example}
		We now illustrate the algorithm by constructing a \emph{strangely ordered periodic orbit} of \emph{rotation pair} $(4,18)$. The given \emph{rotation pair} is of the form $(kp,kq)$ where $k=2$, $p=2$ and $q=9$. To this end, we take two periodic orbits $P_1 = \{ \alpha_1, \alpha_2, \dots,  \alpha_9\}$ and $P_2 = \{ \beta_1, \beta_2, \dots,  \beta_9\}$, each of which exhibits \emph{pattern} $\Gamma_0^{\frac{2}{9}}$. We give the orbits, temporal labeling, that is, $f(\alpha_i) = \alpha_{i+1}$ and $f(\beta_i) = \beta_{i+1}$ for $i \in \{1, 2, \dots 9\}$.  Let the critical points of $P_1$ and $P_2$ be $\alpha_1$ and $\beta_1$ respectively.

		\begin{figure}[H]
			\caption{\emph{Double-ton lifting} of the \emph{pattern} $\Gamma_0^{\frac{2}{9}}$.}
			\centering
			\includegraphics[width=1 \textwidth]{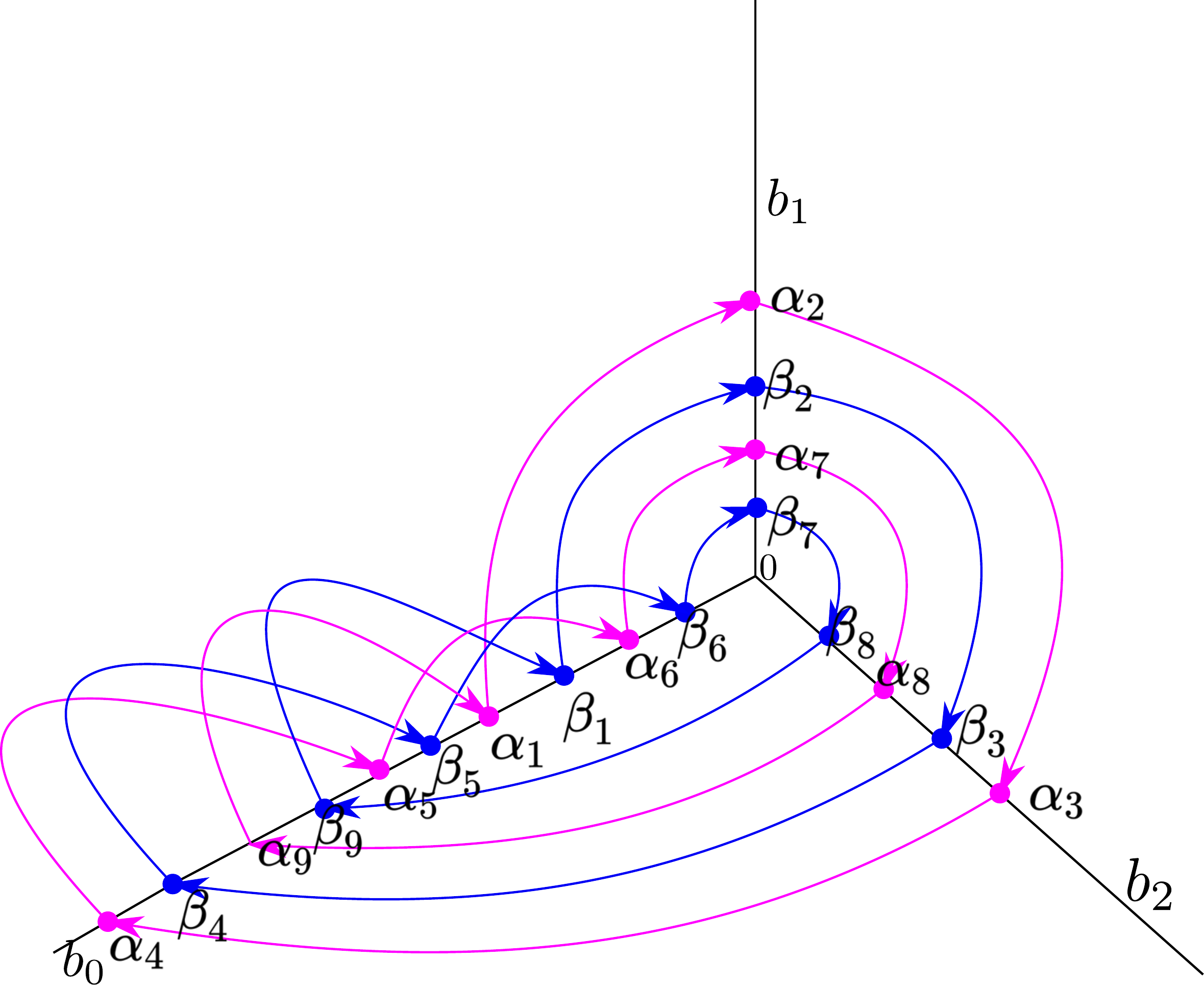}
			\label{rho_less_third_doubleton_lifting}
		\end{figure}

		Let $\mathcal{P} = P_1 \cup P_2$. We place the periodic orbits $P_1$ and $P_2$ in such a manner that the $\mathcal{P}$-\emph{linear map} $f$ has a unique critical point $\alpha_1$ which is the absolute maximum of $f$ and for every $i \in \{1,2, \dots 9\}, \alpha_i > \beta_i$. The set $\mathcal{P}$ with the map $f$ as defined above will be called the \emph{double-ton lifting} of $\Gamma_0^{\frac{2}{9}}$ (See Figure \ref{rho_less_third_doubleton_lifting}).

		Let us now change the map $f$ to obtain a new map $g$ so that the two cycles $P_1$ and $P_2$ gets mingled to one cycle under $g$ but $g$ remains uni-modal and the \emph{code} of the composite periodic orbit $g|_{\mathcal{P}}$ is \emph{non-decreasing}. For this, we note that $\alpha_4$ is the point of $P_1$ farthest away from $a$ in the \emph{branch} $b_0$. Here, $p=2$ and the second point of $P_2$ counting from $\alpha_4$ (which is the point of $P_2$ farthest away from $a$ towards the point $a$) is $\beta_9$. We define, our map $g$ as follows. Define, $g(\alpha_4) = \beta_9$. This means that the point $\beta_8$ which used to be mapped to $\beta_9$ under $f$ originally will now have to be mapped somewhere else under $g$. We set $g(\beta_8) = \alpha_5$. Observe that $\alpha_5$ is the point of $\mathcal{P} = P_1 \cup P_2$ immediately next to $\beta_9$ in the direction towards $a$, that is,  the $f$-image of $\beta_8$ is shifted one place in the direction towards $a$, under the ``new" map $g$. On all other points of $\mathcal{P}$, we keep the actions of the maps $f$ and $g$ exactly the same. In this way, the $f$-orbits $P_1$ and $P_2$ gets amalgamated into one $g$-orbit $\mathcal{P}$ (See Figure \ref{rho_less_third_strangely_ordered}).

		 \begin{figure}[H]
		 	\caption{Unimodal \emph{strangely ordered periodic orbit} of \emph{rotation number} $\frac{2}{9}$ and \emph{rotation pair} $(4,18)$.}
		 	\centering
		 	\includegraphics[width=1 \textwidth]{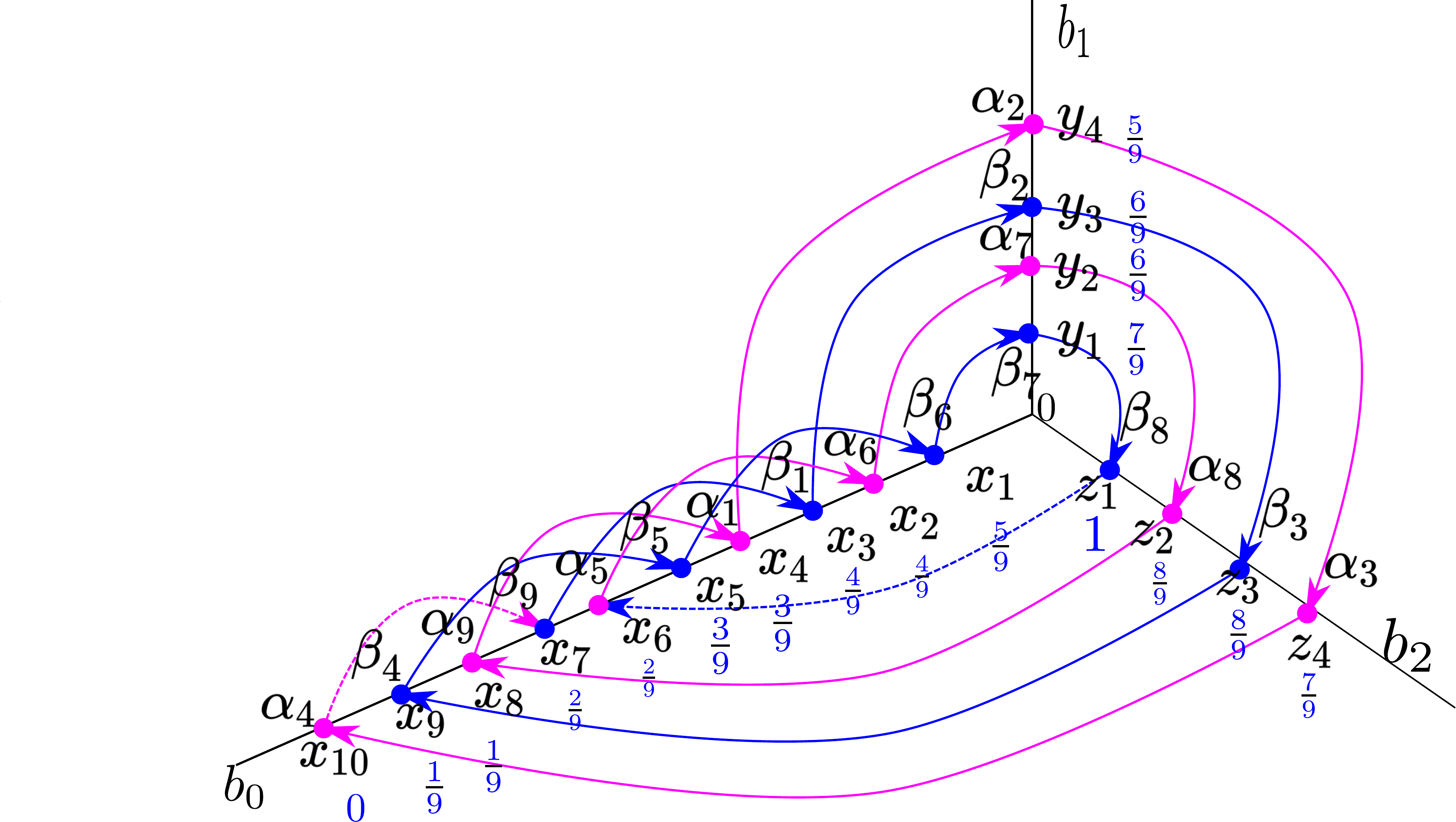}
		 	\label{rho_less_third_strangely_ordered}
		 \end{figure}

		 We now compute the \emph{code} of the orbit $g|_{\mathcal{P}}$. Let us rename the points of $g|_{\mathcal{P}}$ by spatial labeling as follows. The points of $\mathcal{P}$ in the \emph{branch} $b_0$ are named as $x_1, x_2, \dots x_{10}$ in the direction away from $a$ . Similarly, the points of $\mathcal{P}$ in the \emph{branch} .$b_1$ and $b_2$ are renamed as $y_1, y_2,  y_3, y_4$ and $z_1, z_2, z_3, z_4$ respectively, in the direction away from the \emph{branching} point $a$. We set the \emph{code} of the point $x_{10}$, situated farthest away from $a$ in the \emph{branch} $b_0$ to be $0$ and compute the \emph{codes} of the remaining points of $g|_{\mathcal{P}}$ with respect to it. We have $\psi(x_1) = \frac{5}{9}, \psi(x_2) = \frac{4}{9}, \psi(x_3 )= \frac{4}{9}, \psi(x_4) = \frac{3}{9}\, \psi(x_5) = \frac{3}{9}, \psi(x_6) = \frac{2}{9}$, $\psi(x_7) = \frac{2}{9}$, $\psi(x_8) = \frac{1}{9}, \psi(x_9) = \frac{1}{9}$, $\psi(x_{10}) =0$, $\psi(y_1) = \frac{7}{9}, \psi(y_2) = \frac{6}{9}, \psi(y_3) = \frac{6}{9}, \psi(y_4) = \frac{5}{9}, \psi(z_1) = 1, \psi(z_2) = \frac{8}{9} , \psi(z_3) = \frac{8}{9}, \psi(z_4) = \frac{7}{9}$ (See Figure \ref{rho_less_third_strangely_ordered}).  Thus, the \emph{code} for the periodic orbit, $g|_{\mathcal{P}}$ is \emph{non-decreasing}. We observe that $g|_{\mathcal{P}}$ is neither a \emph{triod-twist periodic orbit} nor has a \emph{block structure} over a \emph{triod-twist periodic orbit}. Hence, by Theorem \ref{non:decreasing:conclusion},  $g|_{\mathcal{P}}$ is a \emph{strangely ordered periodic orbit} of \emph{rotation pair} $(4,18)$.

	\end{example}

	We now consider the case, when the \emph{rotation number} $\rho = \frac{r}{s} > \frac{1}{3}$, g.c.d$(r,s)=1$. By \cite{BB5}, there are three distinct uni-modal \emph{triod-twist patterns} $\Delta_j^{\frac{r}{s}}, j \in \{0,1,2\}$ with rotation number $\frac{r}{s}> \frac{1}{3}$. Let $Q_j^{\frac{r}{s}}$ be a periodic orbit which exhibits $\Delta_j^{\frac{r}{s}}$ for some fixed $j \in \{0,1,2\}$. We assume that the \emph{branch}es of $T$ are \emph{canonically ordered} such that the point of $Q_j^{\frac{r}{s}}$ closest to the \emph{branching} point $a$ in each \emph{branch} is \emph{black}. Then the dynamics of $Q_j^{\frac{r}{s}}$ can be described as follows (see \cite{BB5}) (See Figure \ref{rho_greater_third_twist_example}).

	\begin{figure}[H]
		\caption{Dynamics of $\Delta_0^{\frac{2}{5}}$}
		\centering
		\includegraphics[width=0.6 \textwidth]{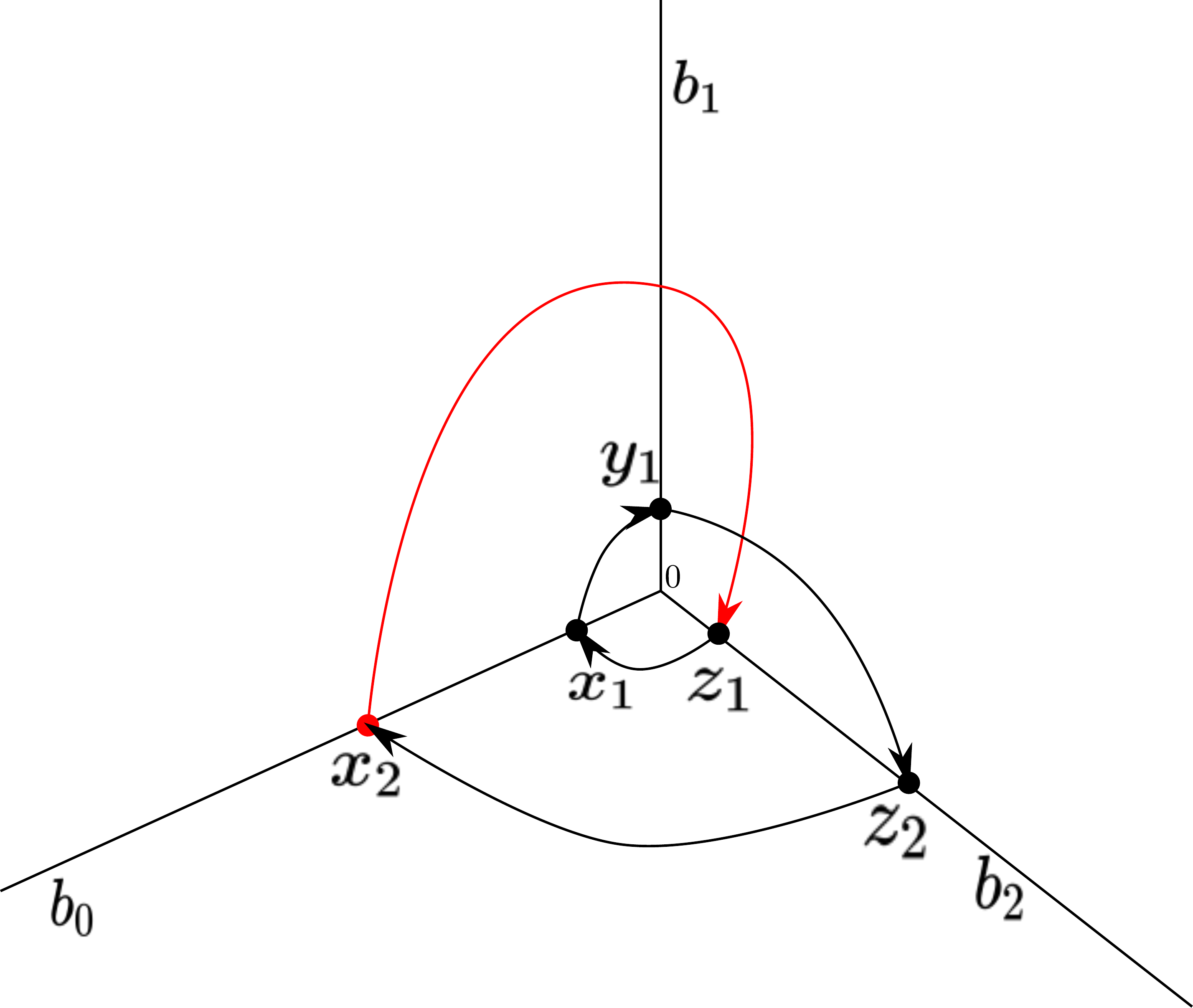}
		\label{rho_greater_third_twist_example}
	\end{figure}

\begin{enumerate}
	\item The \emph{branch} $b_j$ contains $r$ points $x_1, x_2, \dots x_{r}$ numbered in the direction away from $a$. Out of these the first $s-2r$ points $x_1, x_2 \dots x_{s-2r}$ (in the direction away from the \emph{branching} point $a$) are \emph{black} and the remaining $3r-s$ points,  $x_{s-2r+1}, x_{s-2r+2}, $ $  \dots x_{r}$ are \emph{red}. The \emph{branch} $b_{j+1}$ consists of $s-2r$ points $y_1, y_2, \dots $ $  y_{s-2r}$ numbered in the direction away from $a$, all of which are \emph{black}. The \emph{branch} $b_{j+2}$ contains $r$ points $z_1, z_2, \dots z_r$ in the direction away from $a$, all of which are \emph{black}.
	
	\item $f(x_{s-2r+i}) = z_i$ for $i \in \{ 1, 2 \dots 3r-s\}$, that is, in the last $3r-s$  points of the  \emph{branch} $b_j$ (in the direction away from $a$) which are \emph{red} must map to the first $3r-s$ points of the \emph{branch} $b_{j+2}$ (in the direction away from the \emph{branching} point $a$), that is, $f(x_{s-2r+i}) = z_i$ for $i \in \{1,2, \dots 3r-s\}$. 
	
	\item $f(x_i) = y_i$, for $i \in \{1,2, \dots s-2r\}$, that is, the first $s-2r$ points of the \emph{branch} $b_j$ which are \emph{black} (in the direction away from the \emph{branching} point $a$) map to the $s-2r$ \emph{black} points of the \emph{branch} $b_{j+1}$ in an order preserving fashion.
	
	\item Finally, the $s-2r$ \emph{black} points,  $y_1, y_2 \dots y_{s-2r}$ of the \emph{branch} $b_{j+1}$ map to the last $s-2r$ points, $z_{3r-s+1}, z_{3r-s+2}, \dots z_r$ of the \emph{branch} $b_{j+2}$ (in the direction away from the \emph{branching} point $a$), in an order preserving fashion.

\end{enumerate}

We now describe below an algorithm to construct unimodal \emph{strangely ordered periodic orbit} with \emph{rotation pair} $(kr,ks)$ where $k \in \mathbb{N}$. 

\begin{enumerate}
	\item We take $k$ cycles $Q_1, Q_2, \dots, Q_k$, each \emph{exhibiting} the \emph{triod twist pattern} $\Delta_j^{\frac{r}{s}}$ for some fixed $j \in \{0,1,2 \}$ and constructing a $k$-\emph{tuple lifting} of $\Delta_{j}^{\frac{r}{s}}$, similar to the previous case but with a slight modification. Let the cycles $Q_i$, $i \in \{1,2, \dots, k \}$ be assigned a temporal labeling such that the point of the cycle $Q_i$ farthest from the \emph{branching} point $a$ in the \emph{branch} $b_j$ is labeled as $c_{i1}$, where $i \in \{1,2, \dots, k \}$ (note the difference with the previous case). Thus, $Q_i = \{ c_{i1}, c_{i2}, \dots, c_{is} \}$, with $c_{i(t+1)} = f(c_{it})$, for $t \in \{1,2, \dots, s-1 \}$. The orbits $Q_1, Q_2, \dots, Q_k$ are then arranged such that the $\mathcal{Q}$-\emph{linear map} $f$, where $\mathcal{Q} = Q_1 \cup Q_2 \cup \dots \cup Q_k$, is unimodal and satisfies $c_{1t} > c_{2t} > \dots > c_{kt}$, for $t \in \{1,2, \dots, s-1 \}$. Note that the unique critical point of each cycle $Q_i$ is the $(3r-s)$th point of $Q_i$, measured from $c_{i1}$ in the \emph{branch} $b_j$ (towards the \emph{branching} point $a$) (See Figure \ref{rho_greater_third_doubleton_lifting}). 
	
	\item 	We now modify the map $f$ to create a new map $g$ such that the $k$ cycles are merged into a single cycle under $g$ while ensuring $g$ remains unimodal and the \emph{code} of the resulting cycle is \emph{non-decreasing}. Observe that, any $j \in \{0,1,2\}$ can be chosen such that the cycle $Q_i$ exhibits the \emph{pattern} $\Delta_j^{\frac{r}{s}}$. For simplicity, let us select $j=0$, so that all $k$ periodic orbits $Q_1, Q_2, \dots, Q_k$ exhibits pattern $\Delta_0^{\frac{r}{s}}$. Let for each $i \in \{1, 2, \dots, k\}$, $c_{i \ell}$ represent the $(3r-s+1)$-th point of $Q_i$, counting in the direction away from the \emph{branching} point $a$ in the \emph{branch} $b_2$. We define $g(c_{11}) = c_{2\ell}$, that is the point $c_{11}$ of the cycle $Q_1$, which is farthest from the \emph{branching} point $a$ in the \emph{branch} $b_0$ is mapped to the $(3r-s+1)$-th point of the cycle  $Q_2$ in \emph{branch} $b_2$ (in the direction away from the \emph{branching} point $a$) . This adjustment necessitates changes to the image of $c_{2(\ell-1)}$ under $f$, as it previously mapped to $c_{2\ell}$. Instead, we set $g(c_{2(\ell-1)}) $ = $ c_{3\ell}$, effectively shifting its image one step closer to the \emph{branching} point $a$. Similarly, $c_{3(\ell-1)}$, which was mapped to $c_{3\ell}$ under $f$, now has its image moved to $c_{4\ell}$ under the map $g$. This iterative process ensures that for each $j = 2, 3, \dots, k$, the point $c_{j(\ell-1)}$, previously mapped to $c_{j\ell}$ under $f$, is now mapped one step closer to $a$ under $g$, to the point $c_{(j+1)\ell}$. On the final step of this construction, the point $c_{k(\ell-1)}$, previously mapped to $c_{k\ell}$ under $f$, is mapped to $c_{12}$ under $g$. This point, $c_{12}$, is one step closer to $a$ than $c_{k\ell}$ and also coincides with the image of $c_{11}$ under $f$. For all other points in $\mathcal{Q}$, the maps $f$ and $g$ are identical. Through this construction, the $f$-orbits $Q_1, Q_2, \dots, Q_k$ are successfully merged into a single $g$-orbit $\mathcal{Q}$ (See Figure \ref{rho_greater_third_strangely_ordered}).

\item We claim $g|_{\mathcal{Q}}$ is the required unimodal \emph{strangely ordered periodic orbit} of \emph{rotation pair} $(kr, ks)$. Simple computation shows that the \emph{code} for $\mathcal{Q}$ is \emph{non-decreasing}. Also, $\mathcal{Q}$ is neither a \emph{twist periodic orbit}, nor has a \emph{block structure} over a \emph{twist periodic orbit}. So, the claim follows from Theorem \ref{non:decreasing:conclusion}.

\end{enumerate}

	\begin{figure}[H]
	\caption{\emph{Double-ton lifting} of the \emph{pattern} $\Delta_0^{\frac{2}{5}}$}
	\centering
	\includegraphics[width=1 \textwidth]{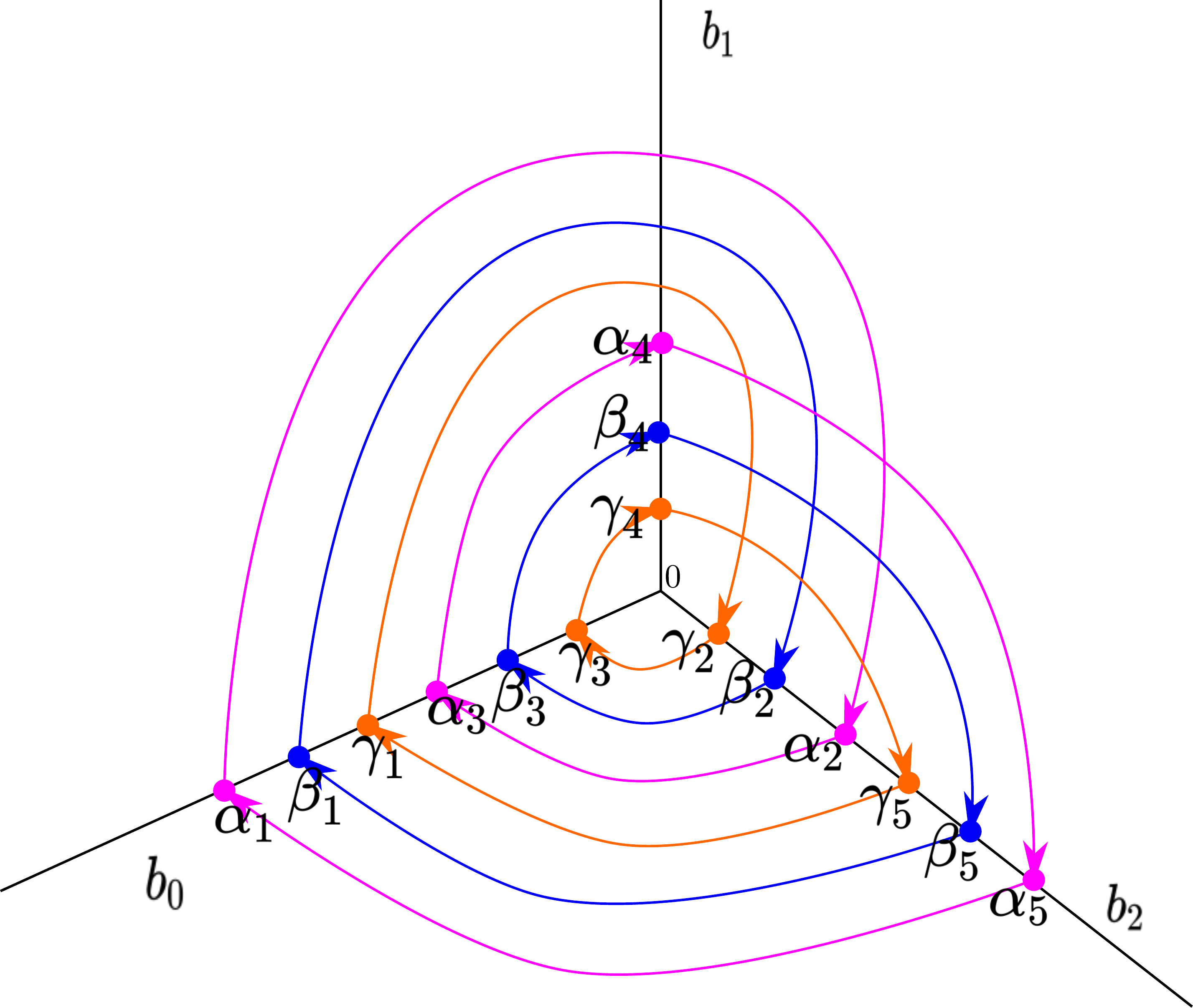}
	\label{rho_greater_third_doubleton_lifting}
\end{figure}

\begin{example}
	We now illustrate the algorithm by constructing a \emph{strangely ordered periodic orbit} of \emph{rotation pair} $(6,15)$. The given \emph{rotation pair} is of the form $(kr,ks)$ where $k=3$, $r=2$ and $s=5$. To this end, we take three periodic orbits $Q_1 = \{ \alpha_1, \alpha_2, \alpha_3, \alpha_4, \alpha_5\}$, $Q_2 = \{ \beta_1, \beta_2, \beta_3, \beta_4, \beta_5\}$ and $Q_3 = \{ \gamma_1, \gamma_2, \gamma_3, \gamma_4, \gamma_5\}$, each of which exhibits \emph{pattern} $\Delta_0^{\frac{2}{5}}$. We give the orbits, temporal labeling, that is, $f(\alpha_i) = \alpha_{i+1}$, $f(\beta_i) = \beta_{i+1}$ and $f(\gamma_i) = \gamma_{i+1}$,  and  for $i \in \{1, 2, 3, 4, 5\}$.  Let $\mathcal{Q} = Q_1 \cup Q_2 \cup Q_3$ and let $f$ be a $\mathcal{Q}$-\emph{linear} map.

	We now form a $3$-\emph{tuple lifting} of the \emph{pattern} $\Delta_0^{\frac{2}{5}}$ (See Figure \ref{rho_greater_third_doubleton_lifting}).  For this, we place the periodic orbits $Q_1, Q_2$ and $Q_3$ in such a manner that (i) $f$ has a unique critical point $\gamma_1$, (ii) $\alpha_1, \beta_1$ and $\gamma_1$ are the points of the cycles $Q_1, Q_2$ and $Q_3$ farthest from the \emph{branching} point $a$ in the \emph{branch} $b_0$ and (iii) $\alpha_i > \beta_i > \gamma_i$ for $i \in \{1,2, 3, 4, 5\}$. 
	Observe that unique critical points of the cycles, $Q_1, Q_2$ and $Q_3$ are respectively $\alpha_1, \beta_1$ and $\gamma_1$. Now, let us modify the map $f$ to form a new map $g$ such that the three cycles $Q_1$, $Q_2$, and $Q_3$ are amalgamated under $g$, while $g$ remains unimodal, and the \emph{code} of the composite periodic orbit $g|_{\mathcal{Q}}$ is \emph{non-decreasing}. Specifically, the $3r-s+1 = (3 \cdot 2 - 5 + 1) = 2$nd point of $Q_2$ in the \emph{branch} $b_2$, counting away from $a$, is $\beta_5$. Set $g(\alpha_1) = \beta_5$. Consequently, $\beta_4$, which was mapped to $\beta_5$ under $f$, must now be mapped elsewhere. Set $g(\beta_4) = \gamma_5$, shifting its image one step closer to the \emph{branching} point $a$. Similarly, $\gamma_4$, originally mapped to $\gamma_5$ under $f$, is now set to $g(\gamma_4) = \alpha_2$, again shifting its image closer to $a$. For all other points, $f$ and $g$ remain identical. In this way, the $f$-orbits $Q_1, Q_2$ and $Q_3$ gets amalgamated into one $g$-orbit $\mathcal{Q}$ (See Figure \ref{rho_greater_third_strangely_ordered}).

	\begin{figure}[H]
		\caption{Unimodal \emph{strangely ordered periodic orbit} of \emph{rotation number} $\frac{2}{5}$ and \emph{rotation pair} $(6,15)$.}
		\centering
		\includegraphics[width=1 \textwidth]{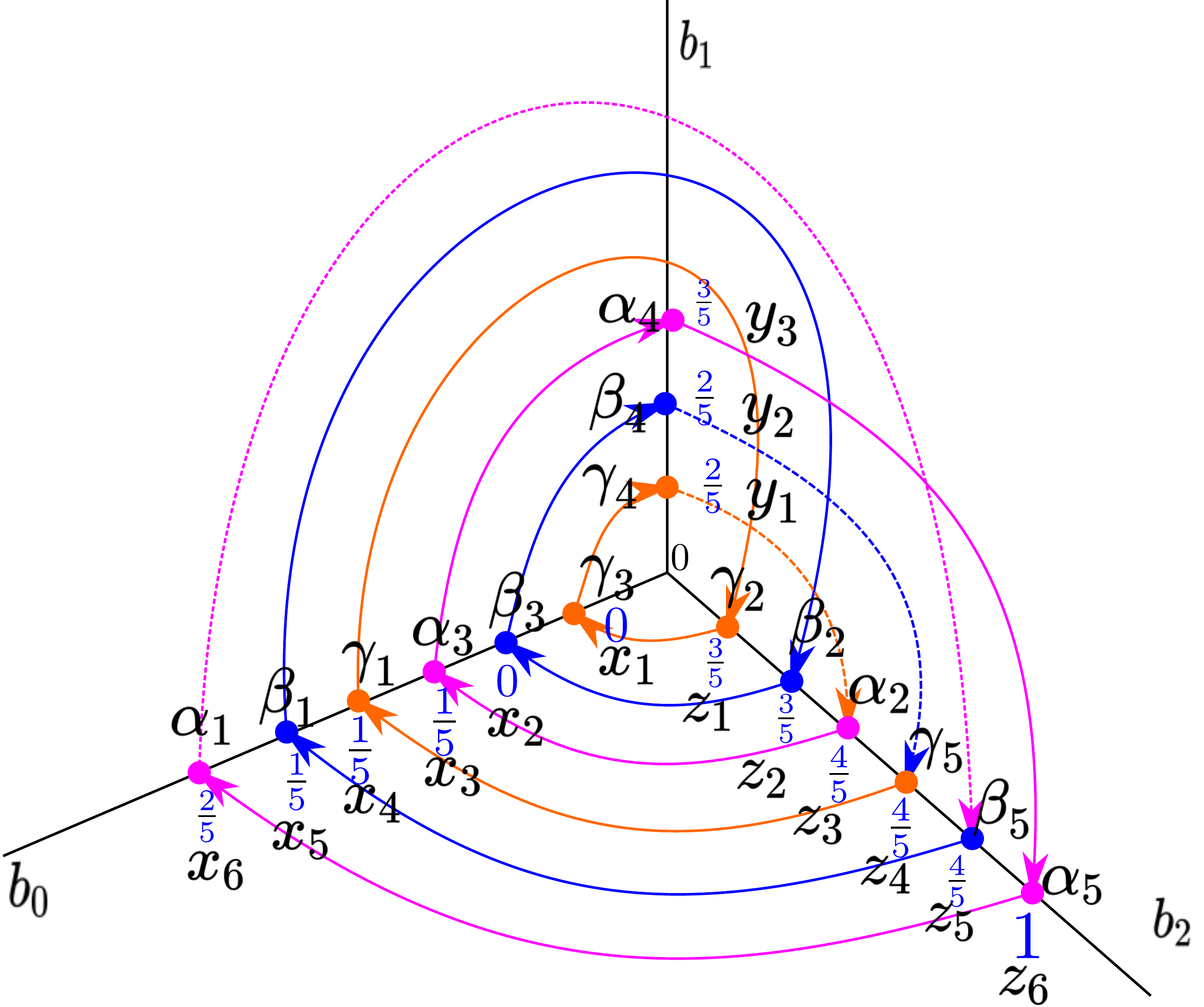}
		\label{rho_greater_third_strangely_ordered}
	\end{figure}

	Let us compute the \emph{code} of the orbit $g|_{\mathcal{Q}}$. Let us rename the points of $g|_{\mathcal{Q}}$ in the \emph{branch} $b_0$ in the spatial labeling as $x_1, x_2, x_3, x_4, x_5, x_6 $ (in the direction away from $a$). Similarly, the points of $\mathcal{Q}$ in $b_1$ in the spatial labeling are named as $y_1, y_2$ and $y_3$ in the direction away from $a$. Finally, the points of $\mathcal{Q}$ in the \emph{branch} $b_2$ are labeled as $z_1, z_2, z_3, z_4, z_5, z_6$ in the direction away from $a$.    We set the \emph{code} of the point $x_{1}$ to be $0$ and compute the \emph{codes} of the remaining points of $g|_{\mathcal{Q}}$ with respect to it. We have,  $ \psi(x_2) = 0, \psi(x_3) = \frac{1}{5},\psi(x_4)= \frac{1}{5}, \psi(x_5) = \frac{1}{5}\, \psi(x_6) = \frac{2}{5}, \psi(y_1) = \frac{2}{5}$, $\psi(y_2) = \frac{2}{5}$, $\psi(y_3) = \frac{3}{5}, \psi(z_1) = \frac{3}{5}$, $\psi(z_2) =\frac{3}{5}$, $\psi(z_3) = \frac{4}{5}, \psi(z_4) = \frac{4}{5}, \psi(z_5) = \frac{4}{5}, \psi(z_6) =1$. Thus, the \emph{code} for the periodic orbit,  $g|_{\mathcal{Q}}$ is \emph{non-decreasing}. We observe that $g|_{\mathcal{Q}}$ is neither a \emph{triod-twist periodic orbit} nor has a \emph{block structure} over a \emph{triod-twist periodic orbit}. Hence, $g|_{\mathcal{Q}}$ is a \emph{strangely ordered} periodic orbit of \emph{rotation pair} $(6,15)$ by Theorem \ref{non:decreasing:conclusion}.

\end{example}

\end{document}